\documentclass[11pt,twoside,english]{article}
\usepackage{amsmath}
\usepackage{amsfonts}
\usepackage{mathrsfs}
\usepackage{color}
\usepackage{ amsmath, amsfonts, amssymb, amsthm, amscd}
\usepackage[T1]{fontenc}
\usepackage[english]{babel}
\usepackage[noinfoline]{imsart}

\newtheorem{theorem}{Theorem}
\newtheorem{lemma}{Lemma}
\newtheorem{proposition}{Proposition}

\newtheorem{example}{Example}
\newtheorem{definition}{Definition}

\newtheorem{remark}{Remark}

\newcommand{\cF}{\ensuremath{\mathcal F}}

\newcommand{\cH}{\ensuremath{\mathcal H}}

\newcommand{\cO}{\ensuremath{\mathcal O}}
\newcommand{\cP}{\ensuremath{\mathcal P}}

\newcommand{\cR}{\ensuremath{\mathcal R}}

\newcommand{\bbN}{{\ensuremath{\mathbb N}} }

\newcommand{\bbR}{{\ensuremath{\mathbb R}} }

\newcommand{\N}{\mathbb{N}}

\newcommand{\be}{\begin{equation}}
\newcommand{\ee}{\end{equation}}
\newcommand{\beq}{\begin{eqnarray}}
\newcommand{\eeq}{\end{eqnarray}}

\newcommand{\Ne}{\N^{\ast}}

\newcommand{\HH}{{\cal H}}

\newcommand{\OO}{{\cal O}}

\newcommand{\R}{\mathbb{R}}

\newcommand{\ced}{\end{proof}}

\newcommand{\p}{^{\prime}}

\begin{document}
\begin{frontmatter}
\title{The Obstacle Problem for Quasilinear Stochastic PDEs with non-homogeneous operator}
\date{}
\runtitle{}
\author{\fnms{Laurent}
 \snm{DENIS}\corref{}\ead[label=e1]{ldenis@univ-evry.fr}}
\thankstext{T1}{The work of the first and third author is supported by the chair \textit{risque de cr\'edit}, F\'ed\'eration bancaire Fran\c{c}aise}
\address{Universit\'e
d'Evry-Val-d'Essonne-FRANCE
\\\printead{e1}}
\author{\fnms{Anis}
 \snm{MATOUSSI}\corref{}\ead[label=e2]{anis.matoussi@univ-lemans.fr}}
\thankstext{t2}{The research of the second author was partially supported by the Chair {\it Financial Risks} of the {\it Risk Foundation} sponsored by Soci\'et\'e G\'en\'erale, the Chair {\it Derivatives of the Future} sponsored by the {F\'ed\'eration Bancaire Fran\c{c}aise}, and the Chair {\it Finance and Sustainable Development} sponsored by EDF and Calyon }
\address{
 LUNAM Université, Université du Maine - FRANCE \\\printead{e2}}

\author{\fnms{Jing}
 \snm{ZHANG}\corref{}\ead[label=e3]{jzhang@univ-evry.fr}}
\address{Universit\'e
d'Evry-Val-d'Essonne
-FRANCE\\\printead{e3}}

\runauthor{L. Denis, A. Matoussi and J. Zhang}

\begin{abstract}
We prove the existence and uniqueness of
solution  of the obstacle problem for quasilinear Stochastic PDEs with non-homogeneous second order operator.
Our method is based on analytical technics coming from the parabolic potential theory. The solution
is expressed as a pair $(u,\nu)$  where $u$ is a predictable
continuous process which takes values in a proper Sobolev space and
$\nu$ is a random regular measure satisfying minimal Skohorod
condition. Moreover, we establish a maximum principle for local solutions of such class of
stochastic PDEs. The proofs are based on a
version of It\^o's formula and estimates for the positive part of a
local solution which is non-positive on the lateral boundary.

\end{abstract}


\end{frontmatter}

%
%
%
%
%
%
%
%
%
%
%
%
%

\section{Introduction}
In this paper we study the following SPDE with obstacle (in short OSPDE):
\begin{equation}\label{SPDEO}\left\{ \begin{split}&du_t(x)= \partial_i  \left(a_{i,j}(t,x)\partial_ju_t(x)  \right) dt +  \partial_i g_i(t,x,u_t(x),\nabla
u_t(x))dt+f(t,x,u_t(x),\nabla u_t(x))dt \\&\quad \ \  \
\ \ \ +\sum_{j=1}^{+\infty}h_j(t,x,u_t(x),\nabla
u_t(x))dB^j_t +\nu (t,dx), \\
&u_t\geq S_t \, , \ \ \\ &u_0=\xi\, .\
\end{split}
\right.
\end{equation} where $a$ is a time-dependant symmetric, uniformly elliptic, measurable matrix defined
on some open domain $\mathcal{O}\subset \bbR^d$, with null Dirichlet condition. The initial
condition is given as $u_0=\xi$, a $L^2(\mathcal{O})-$valued random
variable, and $f$, $g=(g_1,...,g_d)$ and $h=(h_1,...h_i,...)$ are
non-linear random functions. Given an obstacle $S: \Omega\times[0,T]\times \cO \rightarrow \R$, we study the obstacle problem for the SPDE
(\ref{SPDEO}), i.e. we want to find a solution of (\ref{SPDEO})  which
satisfies "$u\geq S$" where the obstacle $S$ is regular in some sense  and
controlled by the solution of a SPDE.\\
In recent work \cite{DMZ12} we have proved  in the homogeneous case, existence and uniqueness of the solution of equation (\ref{SPDEO}) with Dirichlet boundary condition under standard Lipschitz hypotheses  and $L^2$-type integrability conditions on the coefficients. Moreover in \cite{DMZ12max}, still in the homogeneous case, we have obtained  a maximum principle  for local solutions. In these papers we have assumed that $a$ does not depend on time and so many proofs are based on the notion of semigroup associated to the second order operator and on the regularizing property of the semigroup. The aim of this paper is to extend all the results to the non homogeneous case. \\
Let us recall that the solution is  a couple $(u,\nu )$, where $u$ is a process with values in the first order Sobolev space and $\nu$ is a random regular measure forcing $u$ to stay above $S$ and satisfying a minimal Skohorod condition. In order to give a rigorous meaning to the notion of solution, inspired by the works of M. Pierre in the deterministic case (see \cite{Pierre,PIERRE}), we introduce the notion of parabolic capacity.  We construct a solution which admits a quasi continuous version hence defined outside a polar set and use the fact  that regular measures which in general are not absolutely continuous w.r.t. the Lebesgue measure, do not charge polar sets.\\

 There is a huge literature on parabolic SPDE's without obstacle. The study of the $L^p-$norms w.r.t. the randomness of the space-time uniform norm on the trajectories of a stochastic PDE was started by N. V. Krylov in \cite{Krylov} (see also  Kim \cite{Kim}),  for a more complete overview of existing works on this subject see \cite{DMS09,DM11} and the references therein.   Let us also mention that some  maximum principle have been established by  N. V. Krylov \cite{Krylov2}  for linear parabolic spde's on Lipschitz domain.
 Concerning the obstacle problem, there are two approaches, a probabilistic one (see \cite{MatoussiStoica, Klimsiak}) based on the Feynmann-Kac's formula via the backward doubly stochastic differential equations and the analytical one (see \cite{DonatiPardoux,NualartPardoux,XuZhang}) based on the Green function.\\

The main results of this paper are first an existence and uniqueness Theorem for the solution with null Dirichlet condition and a maximum principle for local solutions. This yields for example:
\begin{theorem}
Let $(M_t)_{t\geq 0}$ be an Itô process satisfying some
integrability conditions, $p\geq 2$ and $u$ be a local weak
solution of the obstacle problem (\ref{SPDEO}). Assume that $\partial\cO$ is Lipschitz and $u\leq M$ on $\partial\cO$,
 then for all $t\in [0,T]$: $$ E\left\| \left( u-M\right)
^{+}\right\| _{\infty ,\infty ;t}^p\le k\left( p,t\right) \mathcal{C}(S, f, g, h, M)$$
where $\mathcal{C}(S, f, g, h, M)$ depends only on the barrier $S$, the initial condition $\xi$, coefficients $f,g,h$, the boundary condition  $ M$ and $k$ is a function which only depends
on $p$ and $t$, $\Vert \cdot \Vert _{\infty
,\infty ;t}$ is the uniform norm on $[0,t]\times {\cal O}$.\end{theorem}

\section{Hypotheses and preliminaries}
\subsection{Settings}
Let ${\cal O}$ be an open bounded domain in $\mathbb{{ R}}^d.$ The space $%
L^2\left( {\cal O}\right) $ is the basic Hilbert space of our
framework and
we employ the usual notation for its scalar product and its norm,%
$$ \left( u,v\right) =\int_{{\cal O}}u\left( x\right) v\left(
x\right) dx,\;\left\| u\right\|=\left( \int_{{\cal O}}u^2\left(
x\right) dx\right) ^{\frac 12}. $$ In general, we shall extend the
notation
\[ (u,v)=\int_{\cO} u(x)v(x)\, dx,\]
where $u$, $v$ are measurable functions defined on $\cO$ such that $uv \in L^1 (\cO )$.\\
The first order Sobolev space
of functions vanishing at the
boundary will be denoted as usual by $H_0^1\left( {\cal O}%
\right) .$ Its natural scalar product and norm are%
$$
\left( u,v\right) _{H_0^1\left( {\cal O}\right) }=\left( u,v\right) +\sum_{i=1}^d\int_{%
{\cal O}} \partial _i u\left( x\right)
\partial _iv\left( x\right) dx,\,\left\| u\right\| _{H_0^1\left(
{\cal O}\right) }=\left( \left\| u\right\| _2^2+\left\| \nabla
u\right\| _2^2\right) ^{\frac 12}. $$ We shall denote  by  $
H_{loc}^1 (\cO)$ the space of functions which are locally square
integrable in $ \mathcal{O}$ and which admit first order derivatives
that are also locally square integrable.\\
Another Hilbert space that we use is the second order Sobolev space $H^2_0 (\cO )$ of  functions vanishing at the
boundary and twice differentiable in the weak sense.

We consider a sequence $((B^i(t))_{t\geq0})_{i\in\mathbb{N}^*}$ of
independent Brownian motions defined on a standard filtered
probability space $(\Omega,\mathcal{F},(\mathcal{F}_t)_{t\geq0},P)$
satisfying the usual conditions.

Let $a$ be a measurable and $d\times d$ symmetric matrix defined on $\bbR^+\times\cO$. We assume that
there exist positive constants $\lambda$, $\Lambda$ and $M$ such that for all $\eta\in\bbR^d$ and
almost all $(t,x)\in\bbR^+\times\cO$:
\begin{equation}\label{operator}
\lambda|\eta|^2\leq\sum_{i,j}a_{i,j}(t,x)\eta^i\eta^j\leq\Lambda|\eta|^2\ and\ |a_{i,j}(t,x)|\leq M.
\end{equation}
Let $\triangle=\{(t,x,s,y)\in\bbR^+\times\cO\times\bbR^+\times\cO;\ t>s\}$. We denote by $G:\triangle\rightarrow\bbR^+$ the weak fundamental solution
of the problem
\begin{equation}\label{weak:fundamental}
\partial_tG(t,x;s,y)-\sum_{i=1}^d\partial_ia_{i,j}(t,x)\partial_jG(t,x;s,y)=0
\end{equation}
with Dirichlet boundary condition $G(t,x;s,y)=0$, for all $(t,x)\in(s,+\infty)\times\partial\cO$.

We consider the quasilinear stochastic partial differential equation
(1) with initial condition $u(0,\cdot)=\xi(\cdot)$ and Dirichlet
boundary condition $u(t,x)=0,\ \forall\ (t,x)\in
\bbR^+\times\partial\mathcal{O}$.

We assume that we have predictable random
functions\begin{eqnarray*}&&f:\bbR^+\times\Omega\times\mathcal{O}\times
\bbR\times \bbR^d\rightarrow
\bbR,\\&&g=(g_1,...,g_d):\bbR^+\times\Omega\times\mathcal{O}\times \bbR\times
\bbR^d\rightarrow
\bbR^d,\\&&h=(h_1,...,h_i,...):\bbR^+\times\Omega\times\mathcal{O}\times
\bbR\times \bbR^d\rightarrow \bbR^{\mathbb{N}^*},\end{eqnarray*}In the sequel,
$|\cdot|$ will always denote the underlying Euclidean or $l^2-$norm.
For
example$$|h(t,\omega,x,y,z)|^2=\sum_{i=1}^{+\infty}|h_i(t,\omega,x,y,z)|^2.$$

\textbf{Assumption (H):} There exist non-negative constants $C,\
\alpha,\ \beta$ such that for almost all $\omega$, the following
inequalities hold for all
$(x,y,z,t)\in\mathcal{O}\times\mathbb{R}\times\mathbb{R}^d\times\mathbb{R}^+$:\begin{enumerate}
\item   $|f(t,\omega,x,y,z)-f(t,\omega,x,y',z')|\leq C(|y-y'|+|z-z'|),$
\item $(\sum_{i=1}^d|g_i(t,\omega,x,y,z)-g_i(t,\omega,x,y',z')|^2)^{\frac{1}{2}}\leq
C|y-y'|+\alpha|z-z'|,$
\item $(|h(t,\omega,x,y,z)-h(t,\omega,x,y',z')|^2)^{\frac{1}{2}}\leq
C|y-y'|+\beta|z-z'|,$
\item the contraction property: $2\alpha+\beta^2<2\lambda.$
\end{enumerate}
\begin{remark} This last contraction property ensures existence and uniqueness for the solution of the SPDE without obstacle (see \cite{DM11}).
\end{remark}

Moreover for simplicity, we fix a terminal time $T >0$, we assume that: \\[0.3cm]
 \textbf{Assumption (I):} $$\xi\in L^2(\Omega\times\mathcal{O})\ is\ an\
\mathcal{F}_0-measurable\ random\
variable$$$$f(\cdot,\cdot,\cdot,0,0):=f^0\in
L^2([0,T]\times\Omega\times\mathcal{O};\bbR)$$$$g(\cdot,\cdot,\cdot,0,0):=g^0=(g_1^0,...,g_d^0)\in
L^2([0,T]\times\Omega\times\mathcal{O};\bbR^d)$$$$h(\cdot,\cdot,\cdot,0,0):=h^0=(h_1^0,...,h_i^0,...)\in
L^2([0,T]\times\Omega\times\mathcal{O};\bbR^{\mathbb{N}^*}).$$
We denote by $\cH_T$ the space of $H_0^1(\cO)-$valued predictable $L^2(\cO)-$continuous processes $(u_t)_{t\in[0,T]}$ which satisfy
$$\| u\|_T=E\sup_{t\in[0,T]}\parallel u_t\parallel^2+E\int_0^T\parallel\nabla u_t\parallel^2dt<+\infty.$$
It is the natural space for solutions.\\
The space of test functions is denote by  $\mathcal{D}=\mathcal{C}%
_{c}^{\infty } (\R^+ )\otimes \mathcal{C}_c^2 (\cO )$, where $\mathcal{C}%
_{c}^{\infty } (\R^+ )$ is the space of all real valued infinitely
differentiable  functions with compact support in $\mathbb{R}^+$ and
$\mathcal{C}_c^2 (\cO )$ the set
of $C^2$-functions with compact support in $\cO$.
\subsubsection{Main example of stochastic noise}
\label{mainex} Let $W$ be a
 noise white in time and colored in space, defined on a standard filtered probability
 space  $ \big( \Omega,\cF, (\cF_t)_{t \geq 0}, P \, \big)$ whose covariance function is given by:
 \[ \forall s,t\in\R_+ ,\ \forall x,y\in\cO,\ \ E[\dot{W} (x,s)\dot{W}(y,t)]=\delta (t-s)k(x,y),\]
 where $k:\cO\times \cO  \mapsto \R_+$ is a symmetric and measurable function.\\
 Consider the following SPDE driven by $W$:
 {\small
\begin{equation}
\begin{split}
\label{eW}
  du_t (x)  =
 \big( \sum_{i,j=1}^d \partial_i a_{i,j}(t,x)\partial_j u_t (x)
+f(t,x,u_t (x),\nabla u_t (x)) &+  \sum_{i=1}^d
\partial_i
g_i(t,x,u_t(x),\nabla u_t (x)) \big) dt\\
 &  \hspace{-1 cm}+ \tilde{h} (t,x,u_t(x),\nabla u_t(x))\, W(dt,x),\\
\end{split}
\end{equation}  }
where $f$ and $g$ are as above and $\tilde{h}$ is a random real valued function.\\
We assume that the covariance function $k$ defines a trace class operator
 denoted by $K$ in $L^2 (\mathcal{O})$. It is well known  that there exists an orthogonal
 basis $(e_i )_{i\in \Ne}$ of $L^2 (\cO )$ consisting of eigenfunctions of
 $K$ with corresponding eigenvalues $(\lambda_i )_{i\in\Ne}$ such that
 \[ \sum_{i=1}^{+\infty} \lambda_i <+\infty ,\]
 and
 \[ k(x,y)=\sum_{i=1}^{+\infty} \lambda_i e_i (x)e_i (y).\]
It is also well known that there exists a sequence $((B^i
(t))_{t\geq 0})_{i\in\Ne}$ of
 independent standard Brownian motions such that
 \[ W(dt, \cdot )=\sum_{i=1}^{+\infty}\lambda_i^{1/2} e_i B^i (dt).\]
 So that equation \eqref{eW} is equivalent to equation \eqref{SPDEO} without obstacle and with $h=(h_i)_{i\in\Ne}$ where
 $$\forall i\in\Ne,\ h_i (s,x, y,z)=\sqrt{\lambda_i}\tilde{h}(s,x,y,z)e_i (x).$$
Assume as in \cite{SW} that for all $i\in\Ne$,  $\| e_i \|_{\infty} <+\infty $ and
$$\sum_{i=1}^{+\infty} \lambda_i \| e_i \|_{\infty}^2 <+\infty.$$
Since
 $$  \Big(| h(
t,\omega,x,y,z) -h( t,\omega,x,y{\p},z{\p})
|^2\Big)^{\frac{1}{2}}\leq\left( \sum_{i=1}^{+\infty} \lambda_i \| e_i \|_{\infty}^2\right) \left| \tilde{h} (t,x,y,z)-\tilde{h}(t,x,y{\p},z{\p})\right|^2,$$
$h$ satisfies the Lipschitz hypothesis {\bf (H)-(ii)} if  $\tilde{h}$ satisfies a similar  Lipschitz hypothesis.

\subsection{Parabolic potential analysis}\label{capacity}
In this section we will recall some important definitions and
results concerning the obstacle problem for parabolic PDE in
\cite{Pierre} and \cite{PIERRE}.
\\$\mathcal{K}$ denotes $L^\infty([0,T];L^2(\mathcal{O}))\cap
L^2([0,T];H_0^1(\mathcal{O}))$ equipped with the norm:
\begin{eqnarray*}\parallel
v\parallel^2_\mathcal{K}&=&\parallel
v\parallel^2_{L^\infty([0,T];L^2(\mathcal{O}))}+\parallel
v\parallel^2_{L^2([0,T];H_0^1(\mathcal{O}))}\\
&=&\sup_{t\in[0,T[}\parallel v_t\parallel^2 +\int_0^T \left(
\parallel v_t \parallel^2  +\parallel \nabla v_t\parallel^2\right)\, dt
.\end{eqnarray*} $\mathcal{C}$ denotes the space of continuous
functions on compact support in $[0,T[\times\mathcal{O}$ and
finally:
$$\mathcal {W}=\{\varphi\in L^2([0,T];H_0^1(\mathcal{O}));\ \frac{\partial\varphi}{\partial t}\in
L^2([0,T];H^{-1}(\mathcal{O}))\}, $$ endowed with the
norm$\parallel\varphi\parallel^2_{\mathcal {W}}=\parallel
\varphi\parallel^2_{L^2([0,T];H_0^1(\mathcal{O}))}+\parallel\displaystyle\frac{\partial
\varphi}{\partial t}\parallel^2_{L^2([0,T];H^{-1}(\mathcal{O}))}$. \\
It is known (see \cite{LionsMagenes}) that $\mathcal{W}$ is
continuously embedded in $C([0,T]; L^2 (\cO))$, the set of $L^2 (\cO
)$-valued continuous functions on $[0,T]$. So without ambiguity, we
will also consider
$\mathcal{W}_T=\{\varphi\in\mathcal{W};\varphi(T)=0\}$,
$\mathcal{W}^+=\{\varphi\in\mathcal{W};\varphi\geq0\}$,
$\mathcal{W}_T^+=\mathcal{W}_T\cap\mathcal{W}^+$.\\
We now introduce the notion of parabolic potentials and regular measures which permit to define the parabolic capacity.
\begin{definition}
An element $v\in \mathcal{K}$ is said to be a {\bf parabolic potential} if it satisfies:
$$ \forall\varphi\in\mathcal{W}_T^+,\
\int_0^T-(\frac{\partial\varphi_t}{\partial
t},v_t)dt+\int_0^T\mathcal{E}(\varphi_t,v_t)dt\geq0.$$
We denote by $\mathcal{P}$ the set of all parabolic potentials.
\end{definition}
The next representation property is  crucial:
\begin{proposition}(Proposition 1.1 in \cite{PIERRE})\label{presentation}
Let $v\in\mathcal{P}$, then there exists a unique positive Radon
measure on $[0,T[\times\mathcal{O}$, denoted by $\nu^v$, such that:
$$\forall\varphi\in\mathcal{W}_T\cap\mathcal{C},\ \int_0^T(-\frac{\partial\varphi_t}{\partial t},v_t)dt+\int_0^T\mathcal{E}(\varphi_t,v_t)dt=\int_0^T\int_\mathcal{O}\varphi(t,x)d\nu^v.$$
Moreover, $v$ admits a right-continuous (resp. left-continuous)
version $\hat{v} \ (\makebox{resp. } \bar{v}): [0,T]\mapsto L^2
(\cO)$ .\\
Such a Radon measure, $\nu^v$ is called {\bf a regular measure} and we write:
$$ \nu^v =\frac{\partial v}{\partial t}+Av .$$
\end{proposition}
\begin{remark} As a consequence, we can also define for all $v\in\mathcal{P}$:
$$ v_T =\lim_{t\uparrow T}\bar{v}_t \ \in L^2 (\cO ).$$
\end{remark}

\begin{definition}
Let $K\subset [0,T[\times\mathcal{O}$ be compact, $v\in\mathcal{P}$
is said to be  \textit{$\nu-$superior} than 1 on $K$, if there
exists a sequence $v_n\in\mathcal{P}$ with $v_n\geq1\ a.e.$ on a
neighborhood of $K$ converging to $v$ in
$L^2([0,T];H_0^1(\mathcal{O}))$.
\end{definition}
We denote:$$\mathscr{S}_K=\{v\in\mathcal{P};\ v\ is\ \nu-superior\
to\ 1\ on\ K\}.$$
\begin{proposition}(Proposition 2.1 in \cite{PIERRE})
Let $K\subset [0,T[\times\mathcal{O}$ compact, then $\mathscr{S}_K$
admits a smallest $v_K\in\mathcal{P}$ and the measure $\nu^v_K$
whose support is in $K$ satisfies
$$\int_0^T\int_\mathcal{O}d\nu^v_K=\inf_{v\in\mathcal{P}}\{\int_0^T\int_\mathcal{O}d\nu^v;\ v\in\mathscr{S}_K\}.$$
\end{proposition}
\begin{definition}(Parabolic Capacity)\begin{itemize}
                                        \item Let $K\subset [0,T[\times\mathcal{O}$ be compact, we define
$cap(K)=\int_0^T\int_\mathcal{O}d\nu^v_K$;
                                        \item let $O\subset
[0,T[\times\mathcal{O}$ be open, we define $cap(O)=\sup\{cap(K);\
K\subset O\ compact\}$;
                                        \item   for any borelian
$E\subset [0,T[\times\mathcal{O}$, we define $cap(E)=\inf\{cap(O);\
O\supset E\ open\}$.
                                      \end{itemize}

\end{definition}
\begin{definition}A property is said to hold quasi-everywhere (in short q.e.)
if it holds outside a set of null capacity.
\end{definition}
\begin{definition}(Quasi-continuous)

\noindent A function $u:[0,T[\times\mathcal{O}\rightarrow\mathbb{R}$
 is called quasi-continuous, if there exists a decreasing sequence of open
subsets $O_n$ of $[0,T[\times\mathcal{O}$ with: \begin{enumerate}
                        \item for all $n$, the restriction of $u_n$ to the complement of $O_n$ is
continuous;
                                       \item $\lim_{n\rightarrow+\infty}cap\;(O_n)=0$.
                                     \end{enumerate}
We say that $u$ admits a quasi-continuous version, if there exists
$\tilde{u}$ quasi-continuous  such that $\tilde{u}=u\ a.e.$.
\end{definition}
The next proposition, whose proof may be found in \cite{Pierre} or \cite{PIERRE} shall play an important role in the sequel:
\begin{proposition}\label{Versiont} Let $K\subset \cO$ a compact set, then $\forall t\in [0,T[$
$$cap (\{ t\}\times K)=\lambda_d (K),$$
where $\lambda_d$ is the Lebesgue measure on $\cO$.\\
As a consequence, if $u: [0,T[\times \cO\rightarrow \R$ is a map defined quasi-everywhere then it defines uniquely a map from $[0,T[$ into $L^2 (\cO)$.
In other words, for any $t\in [0,T[$, $u_t$ is defined without any ambiguity as an element in $L^2 (\cO)$.
Moreover, if $u\in \mathcal{P}$, it admits  version $\bar{u}$ which is left continuous on $[0,T]$ with values in $L^2 (\cO )$ so that $u_T =\bar{u}_{T^-}$ is also defined without ambiguity.
\end{proposition}
\begin{remark} The previous proposition applies if for example $u$ is quasi-continuous.
\end{remark}

\begin{proposition}\label{approximation}(Theorem III.1 in \cite{PIERRE})
If $\varphi \in\mathcal{W}$, then it admits a unique quasi-continuous version that we denote by $\tilde{\varphi}$. Moreover, for all $v\in \mathcal{P}$,
the following relation holds:
$$\int_{[0,T[\times \cO} \tilde{\varphi}d\nu^v =\int_0^T \left( -\partial_t \varphi ,v\right)+\mathcal{E}(\varphi ,v)\, dt +\left( \varphi_T ,v_{T}\right) .$$
\end{proposition}

We end this section by a convergence lemma which plays an important role in our approach (Lemma 3.8 in \cite{PIERRE}):
\begin{lemma}\label{convergemeas}
If $v^n\in\mathcal{P}$ is a bounded sequence in $\mathcal{K}$ and
converges weakly to $v$ in $L^2([0,T];H_0^1(\mathcal{O}))$; if $u$ is
a quasi-continuous function and $|u|$ is bounded by a element in
$\mathcal{P}$. Then
$$\lim_{n\rightarrow+\infty}\int_0^T\int_\mathcal{O}ud\nu^{v^n}=\int_0^T\int_\mathcal{O}ud\nu^{v}.$$
\end{lemma}
\begin{remark}For the more general case one can see \cite{PIERRE} Lemma 3.8. \end{remark}

\section{Quasi-continuity of the solution of SPDE without obstacle}\label{quasi-contSPDE}
We consider the  SPDE without obstacle:
\begin{eqnarray}\label{SPDE}du_t(x)&=&\partial_i  \left(a_{i,j}(t,x)\partial_ju_t(x)+g_i(t,x,u_t(x),\nabla
u_t(x))\right)dt+f(t,x,u_t(x),\nabla
u_t(x))dt\nonumber\\&&+\sum_{j=1}^{+\infty}h_j(t,x,u_t(x),\nabla
u_t(x))dB^j_t, \end{eqnarray}
As a consequence of well-known results (see for example  \cite{DM11},
Theorem 11),  we know that under  assumptions {\bf(H)} and {\bf (I)},
SPDE (\ref{SPDE}) with zero Dirichlet boundary condition, admits a
unique  solution  in $\mathcal{H}_T$, we denote it by
$\mathcal{U}(\xi,f,g,h)$, moreover it satisfies the following estimate:
\begin{equation}
\label{priori:estimate}
E[\left\| u\right\| _{T}^{2}]\leq cE\left[ \left\| \xi\right\|
^{2}+\int_{0}^{T}\left( \left\|  f^0_{t}\right\|^{2}+\left\|
|g^0_{t} |\right\| ^{2}+ \left\| |h^0_{t}|\right\| ^{2}\right) dt\right]
\end{equation}
The main  theorem of this section is
the following:\begin{theorem}\label{mainquasicontinuity} Under
assumptions {\bf(H)} and {\bf (I)}, $u=\mathcal{U}(\xi,f,g,h)$ the
solution of SPDE (\ref{SPDE})  admits a quasi-continuous version
denoted by $\tilde{u}$ i.e.
 $u=\tilde{u}$ $dP\otimes dt\otimes dx -$a.e. and for almost all $w\in\Omega$,
 $(t,x)\rightarrow \tilde{u}_t
(w,x)$ is quasi-continuous.
\end{theorem}

Before giving the proof of this theorem,
we need the following lemmas. The first one is proved in \cite{PIERRE}, Lemma 3.3:
\begin{lemma}\label{cap}There exists $C>0$ such that, for all open set $\vartheta\subset [0,T[\times \mathcal{O}$ and  $v\in\mathcal{P}$ with $v\geq1\ a.e.$ on $\vartheta$:$$cap\vartheta\leq C\parallel v\parallel^2_{\mathcal{K}}.$$\end{lemma}
Let $\kappa=\kappa(u,u^+(0))$ be defined as following
$$\kappa=ess\inf\{v\in\cP;\ v\geq u\ a.e.\ and\ v(0)\geq u^+(0)\}.$$
One has to note that $\kappa$ is a random function. From now on, we always take for $\kappa$ the following measurable version
$$\kappa =\sup_n v^n,$$
where  $(v^n)$ is the non-decreasing sequence of random functions given by
\begin{equation} \label{eq:1}
\left\{ \begin{split}
         &\frac{\partial v_t^n}{\partial t}=Lv_t^n+n(v_t^n-u_t)^-\\
                  & v^n_0=u^+(0).
                          \end{split} \right.
                          \end{equation}
From F.Mignot and J.P.Puel \cite{MignotPuel}, we know that for almost all $w\in\Omega$, $v^n (w)$ converges weakly to $v(w)=\kappa(u(w),u^+(0)(w))$ in
$L^2([0,T];H_0^1(\mathcal{O}))$ and that $v\geq u$.
\begin{lemma}\label{estimoftau} We have the following estimate:
\begin{eqnarray*}E\parallel\kappa\parallel_{\mathcal{K}}^2 \, \leq C\left(E\parallel u_0^+\parallel^2+E\parallel u_0\parallel^2+E\int_0^T\parallel f_t^0\parallel^2+\parallel|g_t^0|\parallel^2+\parallel|h_t^0|\parallel^2dt\right),\end{eqnarray*}
where $C$ is a constant depending only on the structure constants of the equation.\\
\end{lemma}
Thanks to (\ref{operator}), the proof of Lemma 3 in \cite{DMZ12} can be easily extended to the case of non-homogeneous operator. \\

{\bf Proof of Theorem \ref{mainquasicontinuity}:}
First of all, we remark that we only need to prove this result in the linear case, namely we consider that $f$, $g$ and $h$ only depend on $t$, $x$ and $\omega$. Then, we approximate the coefficients, the domain and the second order operator in the following way:
\begin{enumerate}
\item We mollify coefficients $a_{i,j}$ and so consider
sequences $(a^n_{i,j})_n$ of $C^{\infty}$ functions such that for
all $n\in\Ne$, the matrix $a^n$ satisfies the same ellipticity and
boundedness assumptions as $a$ and
\[ \forall 1\leq i,j\leq d ,\ \lim_{n\rightarrow +\infty}
a^n_{i,j}= a_{i,j}\   a.e.\]
\item We  approximate $\cO$ by  an increasing sequence of smooth domains $(\cO^n )_{n\geq 1}$.
\item We consider  a sequence $(\xi^n )$ in $C_c^{\infty}
({\cO})$ which converges to $\xi$ in $L^2 (\cO )$ and such that for all $n$, ${\rm supp\,} \xi^n \subset \cO^n$.
\item For each $i\in\Ne$, we construct a sequence of predictable functions $(h^n_i
)$  in \\$\left(L^2 (\Omega )\otimes C_c ([0,+\infty
))\otimes C_c^{\infty}({\cO})\right)$ which converges in
$L^2_{loc} (\R_+ ;L^2 (\Omega \times \cO))$ to $h_i$ such that for all $n$, ${\rm supp\,} w^n_i \subset \cO^n$ and
$$\forall t\geq 0 ,\ E[ \int_0^t \| h^n_{i,s}\|^2 \, ds] \leq E[ \int_0^t \| h_{i,s}\|^2 \, ds],$$
so that
$$E[\int_0^t \| |h^n_s|\|^2 ds ]\leq E[\int_0^t \| |h_s|\|^2 ds ]<+\infty.$$
\item We consider a sequence  of predictable functions $(f^{n})$  in $\left(L^2 (\Omega )\otimes C_c ([0,+\infty
))\otimes C_c^{\infty}({\cO})\right)$ which converges in
$L^2_{loc} (\R_+ ;L^2 (\Omega \times \cO))$ to  $f$
and such that for all $n$, ${\rm supp \,} f^{n} \subset \cO^n$.
\item Finally,  let $(g^{n})$ be a sequence in $\left(L^2
(\Omega )\otimes C_c ([0,+\infty ))\otimes
C_c^{\infty}({\cO})\right)^d  $ which converges in
$L^2_{loc} (\R_+ ;L^2 (\Omega \times \cO)^d)$ to $g$and such that for all $n$, ${\rm supp\,} g^{n} \subset \cO^n$ .
\end{enumerate}
For all
$n\in\Ne$, we put $\Delta^n  =\{(t,x,s,y)\in\R_+ \times \cO^n\times \R_+ \times \cO^n; t> s\}.$ We denote  by $G^n :\Delta^n \mapsto \R_+$ the
  weak fundamental solution of the problem \eqref{weak:fundamental} associated to $a^n$ and
  $\cO^n$:
\begin{equation}
\partial_t G^n (t,x;s,y)  -
\sum_{i,j=1}^d \partial_i a^n_{i,j}(t,x)\partial_j G^n (t,x;s,y) =0
\end{equation}
with  Dirichlet boundary condition $ G^n(t,x,s,y) = 0 , \quad
\mbox{for all} \; (t,x) \in \; (s,\, +\infty ) \times \partial
\cO^n
\,$.\\
In a natural way we extend $G^n$ on $\Delta$ by setting:
$G^n\equiv0$ on $\Delta \setminus \Delta^n$.\\
We define the process $u^n$ by setting for all
$(t,x)\in \R_+ \times \cO$:

\begin{equation}\begin{split}
u_t ^n(\cdot )=&\int_{\cO}G(t,\cdot,0,y)\xi^n(y)\, dy
+\int_0^t\int_{\cO} G(t,\cdot,s,y)f^n_s (y)dy
ds \\
&\ \ +\sum_{i=1}^{d}\int_0^t\int_{\cO}
G(t,\cdot,s,y)\partial_{i,y}g^n_i (y)dy ds
\\&\ \  +\sum_{j=1}^{+\infty} \int_0^t\int_{\cO} G(t,\cdot,s,y)h^n_{j,s} (y) dB^j_s\, .\\
\end{split}
\end{equation}
The main point is that there exists a subsequence of $(G^n )_{n\geq 1}$ which converges everywhere to $G$  on $\Delta $, where $G$ still denotes the fundamental solution of  \eqref{weak:fundamental}, see Lemma 7 in \cite{DM11}.
From Proposition 6 in \cite{DM11}, we know that $u^n\in\cH_T$ is the unique weak solution of (\ref{SPDE}).
$G$ is uniformly continuous in space-time variables on any compact away from the diagonal in time ( see Theorem 6
in \cite{Aronson}) and satisfies Gaussian estimates (see Aronson \cite{Aronson3}), this ensures that for all $n\in\bbN^*$, $u^n$ is $P$-almost surely continuous in $(t,x)$.\\
Moreover, since sequences $(f^n)_n$, $(g^n )_n$ and $(h_n )_n$ are uniformly bounded in $L^2$-spaces, as a consequence of estimate \eqref{priori:estimate}, $(u^n)_n$ is bounded in $\mathcal{H}_T$ hence in $L^2 ([0,T]\times \Omega ;H^1_0 (\cO))$, so that we can extract a subsequence $(u^{n_k})_k$ which converges weakly
in $L^2 ([0,T]\times \Omega;H^1_0 (\cO))$ and such that a sequence of convex combinations $(\hat{u}^n)$ of the form $$\hat{u}^n=\sum_{k=1}^{N_n}\alpha_k^nu^{n_k}$$
converges strongly to $u$ in $L^2(\Omega\times[0,T];H_0^1(\mathcal{O}))$.
It is clear that for all $n\in\bbN^*$, $(\hat{u}^n)_n$ is $P-$almost surely continuous in $(t,x)$.

We consider a sequence of random open sets
$$\vartheta_n=\{|\hat{u}^{n+1}-\hat{u}^n|>\epsilon_n\},\quad\Theta_p=\bigcup_{n=p}^{+\infty}\vartheta_n.$$
Let
$\kappa_n=\kappa(\frac{1}{\epsilon_n }(\hat{u}^{n+1}-\hat{u}^n),\frac{1}{\epsilon_n }(\hat{u}^{n+1}-\hat{u}^n)^+(0))+\kappa(-\frac{1}{\epsilon_n }(\hat{u}^{n+1}-\hat{u}^n),\frac{1}{\epsilon_n }(\hat{u}^{n+1}-\hat{u}^n)^-(0))$,
from the definition of $\kappa$ and the relation (see \cite{PIERRE})
$$\kappa(|v|)\leq\kappa(v,v^+(0))+\kappa(-v,v^-(0)),$$ we know that
 $\kappa_n$ satisfy the conditions of Lemma \ref{cap}, i.e. $\kappa_n\in\mathcal{P}$ et $\kappa_n\geq1\ a.e.$ on $\vartheta_n$,
thus we get the following relation
$$cap\, (\Theta_p )\leq\sum_{n=p}^{+\infty}cap\, (\vartheta_n )\leq\sum_{n=p}^{+\infty}\parallel\kappa_n\parallel^2_{\mathcal{K}}.$$

Thus, remarking that $\hat{u}^{n+1}-\hat{u}^n =\mathcal{U}(\xi^{n+1}-\xi^n ,f^{n+1}-f^n ,g^{n+1}-g^n, h^{n+1}-h^n )$, we apply Lemma \ref{estimoftau} to
$\kappa(\frac{1}{\epsilon_n }(\hat{u}^{n+1}-\hat{u}^n),\frac{1}{\epsilon_n }(\hat{u}^{n+1}-\hat{u}^n)^+(0))$ and
$\kappa(-\frac{1}{\epsilon_n }(\hat{u}^{n+1}-\hat{u}^n),\frac{1}{\epsilon_n }(\hat{u}^{n+1}-\hat{u}^n)^-(0))$ and obtain:
\begin{eqnarray*}E[cap\; ( \Theta_p)]\leq\sum_{n=p}^{+\infty}
E\parallel\kappa_n\parallel^2_{\mathcal{K}}&\leq&
2C\sum_{n=p}^{+\infty}\frac{1}{\epsilon_n^2}(E\parallel
\xi^{n+1}-\xi^n\parallel^2+E\int_0^T\parallel
f_t^{n+1}-f_t^n\parallel^2\\&+&\parallel
|g_t^{n+1}-g_t^n|\parallel^2+\parallel
|h_t^{n+1}-h_t^n|\parallel^2dt).\end{eqnarray*}

Then, by extracting a subsequence, we can consider that
\begin{eqnarray*}E\parallel \xi^{n+1}-\xi^n\parallel^2+E\int_0^T\parallel f_t^{n+1}-f_t^n\parallel^2+\parallel |g_t^{n+1}-g_t^n|\parallel^2+\parallel |h_t^{n+1}-h_t^n|\parallel^2dt\leq\frac{1}{2^n}.\end{eqnarray*}
Then we take $\epsilon_n =\frac{1}{n^2}$ to get
\begin{eqnarray*}E[cap\; (\Theta_p)]\leq\sum_{n=p}^{+\infty}\frac{2Cn^4}{2^n}.\end{eqnarray*}
Therefore\begin{eqnarray*}\lim_{p\rightarrow+\infty}E[cap\; (\Theta_p)]=0.\end{eqnarray*}
For almost all $\omega\in\Omega$, $\hat{u}^n(\omega)$ is continuous in
$(t,x)$ on $(\Theta_p (w))^c$ and $(\hat{u}^n(\omega))_n $ converges uniformly to
$u$ on $(\Theta_p (w))^c$ for all $p$, hence, $u(\omega)$ is continuous
in $(t,x)$ on $(\Theta_p (w))^c$, then from the definition of
quasi-continuous, we know that $u(\omega)$ admits a quasi-continuous
version since $cap \; (\Theta_p )$ tends to $0$ almost surely as $p$ tends to $+\infty$. \hfill $\Box$

\section{Existence and uniqueness result}
From now on, similarly to the homogeneous case studied in \cite{DMZ12}, we make the following assumptions on the obstacle:\\
\textbf{Assumption (O):} The obstacle $S$ is assumed to be an adapted process, quasi-continuous, such that $S_0 \leq \xi$ $P$-almost surely and
controlled by the solution of a SPDE, i.e. $\forall t\in[0,T]$,
\begin{equation}S_t\leq S'_t\end{equation} where
$S'$ is the solution of a linear SPDE
\begin{equation}\left\{\begin{array}{ccl} \label{obstacle}
 dS'_t(x)&=&\partial_i(a_{i,j}(t,x)\partial_jS'_t(x)+g'_{i,t}(x))dt+f'_t(x)dt++\sum_{j=1}^{+\infty}h'_{j,t}(x)dB^j_t\\
                  S'(0)&=&S'_0
\end{array}\right. \end{equation}
with $S'_0\in L^2 (\Omega\times \cO)$
$\mathcal{F}_0$-measurable, $f'$, $g'$ and $h'$  adapted
processes respectively in $L^2
([0,T]\times\Omega\times\cO;\mathbb{R})$,  $L^2
([0,T]\times\Omega\times\cO;\mathbb{R}^d)$ and $L^2
([0,T]\times\Omega\times\cO;\mathbb{R}^{\mathbb{N}^*})$.
\begin{remark}Here again, we know that $S'$ uniquely exists and satisfies the following estimate:
\begin{equation}\label{estimobstacle}E\sup_{t\in[0,T]}\parallel S'_t\parallel^2+E\int_0^T\parallel \nabla S'_t\parallel^2dt\leq
CE\left[\parallel S'_0\parallel^2+\int_0^T(\parallel f'_t\parallel^2+\parallel |g'_t|\parallel^2+\parallel |h'_t|\parallel^2)dt\right].\end{equation}
Moreover, from Theorem \ref{mainquasicontinuity}, $S'$ admits a
quasi-continuous version. \\
Let us also remark that even if this assumption seems restrictive since $S'$ is driven by the same operator and Brownian motions as $u$, it encompasses
a large class of examples. \end{remark}
We now are able to define rigorously the notion of solution to the problem with obstacle:
\begin{definition} A pair
$(u,\nu)$ is said to be a solution of the obstacle problem for
(\ref{SPDE}) if
\begin{enumerate}
    \item $u\in\mathcal{H}_T$ and $u(t,x)\geq S(t,x),\ dP\otimes dt\otimes
    dx-a.e.$ and $u_0(x)=\xi,\ dP\otimes dx-a.e.$;
    \item $\nu$ is a random regular measure defined on
    $[0,T)\times\mathcal{O}$;
    \item the following relation holds almost surely, for all
    $t\in[0,T]$ and $\forall\varphi\in\mathcal{D}$,
     \begin{eqnarray}\label{solution}&&(u_t,\varphi_t)-(\xi,\varphi_0)-\int_0^t(u_s,\partial_s\varphi_s)ds-\sum_{i,j=1}^d\int_0^t\int_{\cO}a_{i,j}(s,x)\partial_i u_{s}(x)\partial_j\varphi _{s}(x)dx\, ds
     \nonumber\\&&=\int_0^t(f_s(u_s,\nabla u_s),\varphi_s)ds+\sum_{i=1}^d\int_0^t(g^i_s(u_s,\nabla u_s),\partial_i\varphi_s)ds\nonumber\\&&
    \quad+\sum_{j=1}^{+\infty}\int_0^t(h^j_s(u_s,\nabla u_s),\varphi_s)dB^j_s+\int_0^t\int_{\mathcal{O}}\varphi_s(x)\nu(dx,ds).\end{eqnarray}
    \item $u$ admits a quasi-continuous version, $\tilde{u}$, and we have  $$\int_0^T\int_{\mathcal{O}}(\tilde{u}(s,x)-{S}(s,x))\nu(dx,ds)=0\ \
    a.s.$$
  \end{enumerate}
\end{definition}
The first important result of this paper is:
\begin{theorem}{\label{maintheo}}
Under assumptions {\bf (H)}, {\bf (I)} and {\bf (O)}, there exists a
unique weak solution of the obstacle problem for the SPDE
(\ref{SPDE}) associated to $(\xi,\ f,\ g,\ h,\ S)$.\\
We denote by $\mathcal{R}(\xi,f,g,h,S)$ the solution of SPDE
(\ref{SPDE}) with obstacle when it exists and is unique.
\end{theorem}
\begin{proof} As we have It\^o's formula and comparison theorem for the solution of non homogeneous SPDE (\ref{SPDE}), see Proposition 9 and Theorem 16 in \cite{DM11},  we can make the same proof as  in the homogeneous case (see \cite{DMZ12}). More precisely, we first establish the result in the linear case  by following Section 5.2 in \cite{DMZ12}. Then, we prove an Itô formula for the difference of two  (linear) solutions of SPDE's with obstacle similarly to Section 5.4 in \cite{DMZ12} and finally conclude thanks to a Picard iteration procedure as in Section 5.5 in \cite{DMZ12}.
\end{proof}

We can also establish the following It\^o formula and comparison theorem for the solution of SPDE (\ref{SPDE}) with obstacle. Here again, the proofs are the same as in \cite{DMZ12}.
\begin{theorem}\label{Itoformula}
Let $u$ be the solution of OSPDE (\ref{SPDE})  and
$\Phi:\mathbb{R}^+\times\mathbb{R}\rightarrow\mathbb{R}$ be a
function of class $\mathcal{C}^{1,2}$. We denote by $\Phi'$ and
$\Phi''$ the derivatives of $\Phi$ with respect to the space
variables and by $\frac{\partial\Phi}{\partial t}$ the partial
derivative with respect to time. We assume that these derivatives
are bounded and $\Phi'(t,0)=0$ for all $t\geq0$. Then for
all $t\in[0,T]$,
\begin{eqnarray*}
&&\int_\mathcal{O}\Phi(t,u_t(x))dx+\int_0^t\int_\cO a_{i,j}(s,x)\Phi''(s,u_s(x))\partial_iu_s(x)\partial_ju_s(x)dxds=\int_\mathcal{O}\Phi(0,\xi(x))dx
\\&&+\int_0^t\int_\mathcal{O}\frac{\partial\Phi}{\partial
s}(s,u_s(x))dxds+\int_0^t(\Phi'(s,u_s),f_s)ds
-\sum_{i=1}^d\int_0^t\int_\mathcal{O}\Phi''(s,u_s(x))\partial_iu_s(x)g_i(x)dxds\\&&+\sum_{j=1}^{+\infty}\int_0^t(\Phi'(s,u_s),h_j)dB_s^j
+\frac{1}{2}\sum_{j=1}^{+\infty}\int_0^t\int_\mathcal{O}\Phi''(s,u_s(x))(h_{j,s}(x))^2dxds\\&&+\int_0^t\int_\mathcal{O}\Phi'(s,\tilde{u}_s(x))\nu(dxds)\,, \qquad\qquad P-a.s.
\end{eqnarray*}
\end{theorem}

This Itô formula naturally leads to a comparison theorem, the proof being the same as in the homogeneous case (see Theorem 8 in \cite{DMZ12}). More precisely,
consider $(u^1,\nu^1 )=\mathcal{R} (\xi^1, f^1,g,h,S^1)$ the solution of the SPDE with obstacle
\begin{equation}\left\{ \begin{split}&du^1_t(x)=\partial_i(a_{i,j}(t,x)\partial_ju^1_t(x)+g_i(t,x,u^1_t(x),\nabla
u^1_t(x))dt+f^1(t,x,u^1_t(x),\nabla
u^1_t(x))dt\nonumber\\&\ \ \ \ \quad \quad+\sum_{j=1}^{+\infty}h_j(t,x,u^1_t(x),\nabla
u^1_t(x))dB^j_t +\nu^1 (x, dt)\\
&u^1\geq S^1\ ,\  u^1_0=\xi^1 ,\  \end{split}\right.\end{equation}
where we assume $(\xi^1, f^1 ,g,h)$ satisfy hypotheses {\bf (H)}, {\bf (I)} and {\bf (O)}. \\
We consider another coefficients $f^2$ which satisfies the same assumptions as $f^1$, another obstacle $S^2$ which satisfies {\bf (O)} and another initial condition $\xi^2$ belonging to $L^2 (\Omega \times \cO)$ and $\mathcal{F}_0$ adapted such that $\xi^2\geq S^2_0$. We denote by $(u^2 ,\nu^2 )=\mathcal{R} (\xi^2, f^2,g,h,S^2)$.
%
\begin{theorem} \label{comparison}Assume that the following conditions
hold\begin{enumerate}
      \item $\xi^1\leq\xi^2,\ dx\otimes dP-a.e.$
      \item $f^1(u^1,\nabla u^1)\leq f^2(u^1,\nabla u^1),\ dt\otimes dx\otimes dP-a.e.$
      \item $S^1\leq S^2,\ dt\otimes dx\otimes dP-a.e.$
    \end{enumerate}
Then for almost all $\omega\in\Omega$, $u^1(t,x)\leq u^2(t,x)\
q.e.$\end{theorem}
\section{Maximum principle for local solutions of the OSPDE}
\subsection{$L^{p,q}$-spaces}
\label{lpq}
For each $t>0$ and for all real numbers $p,\,q\geq 1$, we denote by $%
L^{p,q}([0,t]\times {\cal O})$ the space of (classes of) measurable
functions $u:[0,t]\times {\cal O}\longrightarrow \mathbb{{R}}$ such
that
$$
\Vert u\Vert _{p,q;\,t}:=\left( \int_0^t\left( \int_{{\cal O}%
}|u(s,x)|^p\,dx\right) ^{q/p}\,ds\right) ^{1/q} $$ is finite. The
limiting cases with $p$ or $q$ taking the value $\infty $ are also
considered with the use of the essential sup norm.

The space of  measurable functions $u:\mathbb{{R}}_{+}\rightarrow
L^2\left(
{\cal O}\right) $ such that $\left\| u\right\| _{2,2;t}<\infty ,$ for each $%
t\ge 0,$ is denoted by $L_{loc}^2\left( \mathbb{{R}}_{+};L^2\left( {\cal O}%
\right) \right) ,$ where $\R_+$ denotes the set of non-negative real numbers. Similarly, the space $L_{loc}^2\left( \mathbb{{R}}%
_{+};H_0^1\left( {\cal O}\right) \right) $ consists of all
measurable
functions $u: \mathbb{{R}}_{+}\rightarrow H_0^1\left( {\cal O}\right) $ such that%
$$ \left\| u\right\| _{2,2;t}+\left\| \nabla u\right\|
_{2,2;t}<\infty , $$ for any $t\ge 0.$\\

 We recall that the Sobolev inequality states that%
$$ \left\| u\right\| _{2^{*}}\le c_S\left\| \nabla u\right\| _2,
$$ for each $u\in H_0^1\left( {\cal O}\right) ,$ where $c_S>0$ is
a constant
that depends on the dimension and $2^{*}=\frac{2d}{d-2}$ if $d>2,$ while $%
2^{*}$ may be any number in $]2,\infty [$ if $d=2$ and $2^{*}=\infty $ if $%
d=1.$\\ Finally, we introduce the following norm which is obtained by interpolation in $L^{p,q}$-spaces:
$$ \left\| u\right\| _{\#;t}=\left\|
u\right\| _{2,\infty ;t}\vee \left\| u\right\|
_{2^{*},2;t}, $$
and we denote by $L_{\# ;t}$ the set of functions $u$ such that $\left\| u \right\| _{\#;t}$ is finite. Its dual space is a functional space: $L^*_{\#;t}$ equipped with the norm $\parallel\ \parallel^*_{\#;t}$ and we have
\begin{equation}\label{Holder2} \int_0^t\int_{{\cal O}}u\left( s,x\right) v\left( s,x\right)
dxds\le \left\| u\right\| _{\#;t}\left\| v\right\|_{\#;t}^*, \end{equation} for
any $u\in L_{\#;t}$ and $v\in L_{\#;t}^*. $\\

\subsection{Local solutions}
 We define $\HH_{loc}=\HH_{loc}(\mathcal{O})$ to be the set of
$H^1_{loc} (\OO )$-valued predictable processes defined on $[0,T]$ such that for any
compact subset $K$ in $\OO$:

\[
\left( E \sup_{0\leq s\leq T}\int_K u_{s}(x)^{2}\, dx
+E\int_{0}^{T}\int_K |\nabla u_{s}(x)|^2\, dx ds\right) ^{1/2}\;< \;
\infty .
\]

\begin{definition}
We say that a Radon measure $\nu$ on $[0,T[\times \cO$ is a {local regular measure} if for any  non-negative $\phi$  in $\mathcal{C}_c^\infty(\cO)$,  $\phi\nu$ is a regular measure.
\end{definition}
In \cite{DMZ12max} (see Proposition 2.10), we have proved:
\begin{proposition}
Local regular measures do not charge polar sets (i.e. sets of capacity 0).
\end{proposition}
We can now define the notion of local solution:
\begin{definition} A pair
$(u,\nu)$ is said to be a local solution of the problem
(\ref{SPDEO}) if
\begin{enumerate}
    \item $u\in\mathcal{H}_{loc}$, $u(t,x)\geq S(t,x),\ dP\otimes dt\otimes
    dx-a.e.$ and $u_0(x)=\xi,\ dP\otimes dx-a.e.$;
    \item $\nu$ is a local random regular measure defined on
    $[0,T[\times\mathcal{O}$;
    \item the following relation holds almost surely, for all
    $t\in[0,T]$ and all $\varphi\in\mathcal{D}$,
     \begin{equation}\begin{split}\label{solutionlocal}(u_t,\varphi_t)=&(\xi,\varphi_0)+\int_0^t(u_s,\partial_s\varphi_s)ds-\sum_{i,j}\int_0^t\int_{\cO}a_{i,j}(s,x)\partial_i u_s(x)\partial_j\varphi_s(x)dx\, ds\\&-\sum_{i=1}^d\int_0^t(g^i_s(u_s,\nabla u_s),\partial_i\varphi_s)ds
    +\int_0^t(f_s(u_s,\nabla u_s),\varphi_s)ds\\&+\sum_{j=1}^{+\infty}\int_0^t(h^j_s(u_s,\nabla u_s),\varphi_s)dB^j_s+\int_0^t\int_{\mathcal{O}}\varphi_s(x)\nu(dx,ds).\end{split}\end{equation}
    \item $u$ admits a quasi-continuous version, $\tilde{u}$, and we have  $$\int_0^T\int_\cO(\tilde{u}(s,x)-S(s,x))\nu(dx,ds)=0,\ \
    P-a.s.$$
  \end{enumerate}
\end{definition}
We denote by $\cR_{loc}(\xi,f,g,h,S)$ the set of all the local
solutions $(u,\nu)$.

\subsection{Hypotheses}
In order to get  some $L^p-$ estimates for the uniform
norm  of the positive part of the solution of  (\ref{SPDEO}),  we need stronger integrability conditions on the coefficients and the initial condition. To this end, we consider the following assumptions: for
$p\geq2$:

\textbf{Assumption (HI$\mathbf{2
p}$)}$$E\left(\left\|\xi\right\|_\infty^p+\left\|
f^0\right\|^2_{2,2;T}+\left\| |g^0 |\right\|^2_{2,2;T}+\left\| |h^0
|\right\|^2_{2,2;T}\right)<\infty .$$

\textbf{Assumption (OL):} The obstacle $S: [0,T]\times \Omega\times \cO\rightarrow \R$ is  an
adapted random field, almost surely  quasi-continuous, such that $S_0 \leq \xi$
$P$-almost surely and controlled by a \textbf{local} solution of an
SPDE, i.e. $\forall t\in[0,T],$
\begin{equation*}S_t\leq S'_t,\quad dP\otimes dt\otimes dx-a.e.\end{equation*} where
$S'$ is a \textbf{local} solution (for the definition of local solution see for example Definition 1 in \cite{DMS09}) of the linear SPDE
\begin{equation*}\left\{\begin{array}{ccl}
 dS'_t&=&LS'_tdt+f'_tdt+\sum_{i=1}^d \partial_i g'_{i,t}dt+\sum_{j=1}^{+\infty}h'_{j,t}dB^j_t\\
                  S'(0)&=&S'_0 .
\end{array}\right. \end{equation*}

\textbf{Assumption (HIL)} $$  E \int_K |\xi (x)|^2 dx + E \,
\int_{0}^T \int_K \big( |f_s^0(x)|^2 + |g_s^0 (x)|^2 + |h_s^0(x)|^2
\, \big) dx ds < \infty,
$$for any compact set $K \subset \mathcal{O}$.

\textbf{Assumption (HOL)} $$  E \int_K |S'_0|^2 dx + E \, \int_{0}^T
\int_K \big( |f'_t(x)|^2 + |g'_t (x)|^2 + |h'_t (x)|^2 \, \big) dx
dt < \infty
$$for any compact set $K \subset \mathcal{O}$.

\textbf{Assumption (HO$\mathbf{\infty p}$)}
$$S'_0\in L^\infty(\Omega\times\cO)\ and\ E\left((\left\|
f'\right\|_{\infty,\infty;T})^p+(\left\|  |g'
|^2\right\|_{\infty,\infty;T})^{p/2}+(\left\|  |h'
|^2\right\|_{\infty,\infty;T})^{p/2}\right)<\infty .$$

As our approach is based on some estimates of $u-S'$ that we obtain thanks to the Itô formula, we need to   introduce the following functions: $$\bar{f}(t,\omega,x,y,z)=f(t,\omega,x,y+S'_t,z+\nabla
S'_t)-f'(t,\omega,x)$$
$$\bar{g}(t,\omega,x,y,z)=g(t,\omega,x,y+S'_t,z+\nabla S'_t)-g'(t,\omega,x)$$
$$\bar{h}(t,\omega,x,y,z)=h(t,\omega,x,y+S'_t,z+\nabla S'_t)-h'(t,\omega,x).$$
And we consider: \\
\textbf{Assumption (HD$\mathbf{\theta p}$)} $$E((\left\|
\bar{f}^0\right\|^*_{\theta;T})^p+(\left\|  |\bar{g}^0
|^2\right\|^*_{\theta;T})^{p/2}+(\left\|  |\bar{h}^0
|^2\right\|^*_{\theta;T})^{p/2})<\infty .$$

This assumption is fulfilled in the following case:
\begin{example}
If $\left\|\nabla S'\right\|^*_{\theta;T},\ \left\|f^0\right\|^*_{\theta;T},\ \left\|g^0\right\|^*_{\theta;T}\ and\ \left\|h^0\right\|^*_{\theta;T}$ belong to $ L^p(\Omega,P),
$ and assumptions {\bf (H)} and  {\bf(HO$\mathbf{\infty p}$)}  hold, then:\vspace{0.5cm}
\\$\bar{f}$ satisfies the Lipschitz condition with the same Lipschitz coefficients:
 \begin{eqnarray*}\left|\bar{f}(t,\omega,x,y,z)-\bar{f}(t,\omega,x,y',z')\right|&=&\big|f(t,\omega,x,y+S'_t(x),z+\nabla S'_t(x))+f'(t,\omega,x)\\&-&f(t,\omega,x,y'+S'_t(x),z'+\nabla S'_t(x))-f'(t,\omega,x)\big|\\&\leq& C\left|y-y'\right|+C\left|z-z'\right|.\end{eqnarray*}

$\bar{f}$ satisfies the integrability condition:
\begin{eqnarray*}\left\|\bar{f}^0\right\|^*_{\theta;T}&=&\left\|f(S',\nabla
S')-f'\right\|^*_{\theta;T}\leq \left\|f(S',\nabla
S')\right\|^*_{\theta;T}+\left\|f'\right\|^*_{\theta;T}\\&\leq&\left\|f^0\right\|^*_{\theta;T}+C\left\|S'\right\|^*_{\theta;T}+C\left\|\nabla
S'\right\|^*_{\theta;T}+\left\|f'\right\|_{\infty,\infty;T}.\end{eqnarray*}

And the same for $\bar{g}$ and $\bar{h}$, which proves that
{\bf (HD$\mathbf{\theta p}$)} holds.
\end{example}
\subsection{The main results}
We now introduce the lateral boundary condition that we consider:
\begin{definition}
If $u$ belongs to $\cH_{loc}$, we say that $u$ is non-negative
on the boundary of $\cO$ if $u^+$ belongs to $\cH_T$ and
we denote it simply: $u\leq0$ on $\partial\cO$. More generally, if $M$ is a random field defined on $[0,T]\times \cO$, we note $u\leq M$ on $\partial \cO$ if $u-M\leq 0$ on $\partial\cO$.
\end{definition}

From now on, we can follow step by step the proof of the maximum principle for OSPDE in the homogeneous case in \cite{DMZ12max}: the first step consists in establishing an estimate for the positive part of the solution with null Dirichlet condition. To get this estimate,
we can adapt to our case the arguments of proof of Proposition 5.2 in \cite{DMZ12max}, then Itô formula for the difference of 2 elements in $\mathcal{R}_{loc}(\xi, f,g,h,S)$ (Proposition 5.3 in \cite{DMZ12max}). This yields  the comparison theorem (see Theorem 5.4 in \cite{DMZ12max}):
\begin{theorem}{\label{comparison}}
Assume that $\partial \cO$ is Lipschitz. Let $(\xi^i,f^i,g,h,S^i)$, $i=1,2$, satisfy assumptions {\bf(H)}, {\bf(HIL)}, {\bf(OL)} and {\bf(HOL)}.
 Consider $(u^i,\nu^i)\in {\cal R}_{loc}\left( \xi ^i,f^i,g,h,S^i\right)
,i=1,2$ and suppose that the process $\left( u^1-u^2\right) ^{+}$
belongs to $\HH_T
$ and that one has%
$$ E\left( \left\| f^1\left(.,., u^2,\nabla u^2\right) -f^2\left(.,.,
u^2,\nabla u^2\right) \right\| ^*_{\#;t}\right) ^2<\infty ,\;\;
\mbox{ for all} \quad t\in[0,T].$$ If $\xi ^1\le \xi ^2$ a.s.,
$f^1\left(t,\omega, u^2,\nabla u^2\right) \le f^2\left(t,\omega,
u^2,\nabla u^2\right) $,  $dt\otimes dx \otimes dP$-a.e. and
$S^1\leq S^2$, $dt\otimes dx\otimes dP$-a.s., then one has $u^1
(t,x)\le u^2 (t,x)$, $dt \otimes dx\otimes dP$-a.e.
\end{theorem}
By adapting the proof of Theorem 5.5 in \cite{DMZ12max}, we  get first the maximum principle in the case $u\leq 0$ on $\partial \cO$:
\begin{theorem}
\label{maxprinc} Assume that $\partial \cO$ is Lipschitz and suppose that Assumptions {\bf (H)}, {\bf (OL)},
{\bf (HOL)},{\bf (HI$\mathbf{2 p}$)}, {\bf (HO$\mathbf{\infty p}$)}
and {\bf (HD$\mathbf{\theta p}$)} hold  for some $\theta \in [0,1[$,
$p\geq2$ and that the constants of the Lipschitz conditions satisfy
$$\alpha +\frac{\beta ^2}2+72\beta ^2<\lambda. $$ Let $(u,\nu)\in
{\cal R}_{loc}\left( \xi ,f,g,h,S\right) $ be such that
$u^{+}\in \HH.$  Then one has%
\begin{eqnarray*} E\left\| u^{+}\right\| _{\infty ,\infty ;t}^p&\le &k(t)c(p)E\big(\left\|\xi^+-S'_0\right\|^p_\infty+(\left\| \bar{f}^{0,+}\right\|^*_{\theta;t})^p+
(\left\||\bar{g}^0|^2\right\|^*_{\theta;t})^{\frac{p}{2}}+(\left\||\bar{h}^0|^2\right\|^*_{\theta;t})^{\frac{p}{2}}\\&+&\left\|(S'_0)^+\right\|^p_\infty+(\left\|f^{',+}\right\|^*_{\theta;t})^p+
(\left\||g'|^2\right\|^*_{\theta;t})^{\frac{p}{2}}+(\left\||h'|^2\right\|^*_{\theta;t})^{\frac{p}{2}}\big)\end{eqnarray*}
where $k\left( t\right) $ is constant that depends on the structure
constants and $t\in [0,T].$
\end{theorem}
As in the homogeneous case (see Theorem 5.6 in \cite{DMZ12max}), we can generalize the previous result by considering a real It\^o
process of the
form$$M_t=m+\int_0^tb_sds+\sum_{j=1}^{+\infty}\int_0^t\sigma_{j,s}dB_s^j$$
where $m$ is a random variable and $b=(b_t)_{t\geq0}$,
$\sigma=(\sigma_{1,t},...,\sigma_{n,t},...)_{t\geq0}$ are adapted processes.
\begin{theorem}
\label{maintheo} Assume that $\partial \cO$ is Lipschitz and suppose that Assumptions {\bf (H)}, {\bf (OL)},
{\bf (HOL)},{\bf (HI$\mathbf{2 p}$)}, {\bf (HO$\mathbf{\infty p}$)}
and {\bf (HD$\mathbf{\theta p}$)} hold for some $\theta \in [0,1[$,
$p\geq2$ and that the constants of the Lipschitz conditions satisfy
$$\alpha +\frac{\beta ^2}2+72\beta ^2<\lambda. $$ Assume also that
$m$ and the processes $b$ and $\sigma $
satisfy the following integrability conditions%
$$ E\left| m\right| ^p<\infty ,\;E\left( \int_0^t\left| b_s\right|
^{\frac 1{1-\theta }}ds\right) ^{p\left( 1-\theta \right) }<\infty
,\;E\left(
\int_0^t\left| \sigma _s\right| ^{\frac 2{1-\theta }}ds\right) ^{\frac{%
p\left( 1-\theta \right) }2}<\infty , $$ for each $t\in[0,T].$ Let
$(u,\nu)\in {\cal R}_{loc}\left( \xi ,f,g,h,S\right) $ be such that
$\left( u-M\right) ^{+}$ belongs to $\HH_T$.
Then one has
\begin{eqnarray}\label{MaxiPrinc}E\left\|(u-M)^+\right\|^p_{\infty,\infty;t}&\leq& c(p)k(t)E\big[\left\|(\xi-m)^+-(S'_0-m)\right\|_\infty^p+
\left(\left\|\bar{f}^{0,+}\right\|_{\theta;t}^*\right)^p\nonumber\\&+&\left(\left\|\left|\bar{g}^0\right|^2\right\|_{\theta;t}^*\right)^{\frac{p}{2}}
+\left(\left\|\left|\bar{h}^0\right|^2\right\|_{\theta;t}^*\right)^{\frac{p}{2}}+\left\|(S'_0-m)^+\right\|^p_\infty\\&+&\left(\left\|(f'-b)^+\right\|^*_{\theta;t}\right)^p+\left(\left\||g'|^2\right\|^*_{\theta;t}\right)^{\frac{p}{2}}+\left(\left\||h'-\sigma|^2\right\|^*_{\theta;t}\right)^{\frac{p}{2}}\big]\nonumber\end{eqnarray}
where $k\left( t\right) $ is the constant from the
preceding corollary. The right hand side of this estimate is dominated by the following quantity which is expressed directly in terms of the characteristics of the process $M$,
\begin{eqnarray*}
&&c(p)k(t)E\big[\left\|(\xi-m)^+-(S'_0-m)\right\|^p_\infty+\left(\left\|\bar{f}^{0,+}\right\|_{\theta;t}^*\right)^p
+\left(\left\|\left|\bar{g}^0\right|^2\right\|_{\theta;t}^*\right)^{\frac{p}{2}}+\left(\left\|\left|\bar{h}^0\right|^2\right\|_{\theta;t}^*\right)^{\frac{p}{2}}\\&&
\quad\qquad\quad+\left\|(S'_0-m)^+\right\|^p_\infty+\left(\left\|f^{',+}\right\|_{\theta;t}^*\right)^p
+\left(\left\|\left|g'\right|^2\right\|_{\theta;t}^*\right)^{\frac{p}{2}}+\left(\left\|\left|h'\right|^2\right\|_{\theta;t}^*\right)^{\frac{p}{2}}
\\&&\quad\qquad\quad+\left(\int_0^t\left|b_s\right|^{\frac{1}{1-\theta}}ds\right)^{p(1-\theta)}+\left(\int_0^t\left|\sigma_s\right|^{\frac{2}{1-\theta}}ds\right)^{\frac{p(1-\theta)}{2}}
\big].
\end{eqnarray*}
\end{theorem}

\begin{center}
\begin{minipage}[t]{7cm}
Laurent DENIS \\
Laboratoire d'Analyse et Probabilit\'es\\
 Universit{\'e} d'Evry Val
d'Essonne\\
Rue du P\`ere Jarlan\\
 F-91025 Evry Cedex, FRANCE\\
 e-mail: ldenis{\char'100}univ-evry.fr
\end{minipage}
\hfill
\begin{minipage}[t]{7cm}
 Anis MATOUSSI \\
 LUNAM Université, Université du Maine\\
 Fédération de Recherche 2962 du CNRS\\
 Mathématiques des Pays de Loire \\
Laboratoire Manceau de Mathématiques\\
 Avenue Olivier Messiaen\\ F-72085 Le Mans Cedex 9, France \\
email : anis.matoussi@univ-lemans.fr\\
and \\
CMAP,  Ecole Polytechnique, Palaiseau
\end{minipage}
\hfill \vspace*{0.8cm}
\begin{minipage}[t]{12cm}
Jing ZHANG \\
Laboratoire d'Analyse et Probabilit\'es\\
 Universit{\'e} d'Evry Val
d'Essonne\\
Rue du P\`ere Jarlan\\
 F-91025 Evry Cedex, FRANCE\\
Email: jing.zhang.etu@gmail.com

\end{minipage}
\end{center}

\end{document}